\documentclass[11pt]{amsart}
\usepackage{a4wide}
\usepackage{kotex}
\usepackage{amsmath}
\usepackage{amsfonts}
\usepackage{amscd}

\usepackage{amssymb}
\usepackage{graphicx}
\usepackage{dsfont}
\usepackage{color}
\usepackage{bbm}
\usepackage[colorlinks]{hyperref}
\hypersetup{citecolor=blue,}

\usepackage{sseq}
\usepackage{tikz-cd}
\usepackage[abs]{overpic} 
\usepackage{bm}

\newcommand*{\longhookrightarrow}{\ensuremath{\lhook\joinrel\relbar\joinrel\rightarrow}}

\theoremstyle{theorem}
\newtheorem{TheoremA}{Theorem}

\newtheorem{theorem}{Theorem}[section]
\newtheorem{lemma}[theorem]{Lemma}

\newtheorem{proposition}[theorem]{Proposition}

\newtheorem*{question*}{Question}

\theoremstyle{definition}
\newtheorem*{definition*}{Definition}
\newtheorem{definition}[theorem]{Definition}

\newtheorem*{example*}{Example}
\newtheorem{example}[theorem]{Example}
\newtheorem*{observation*}{Observation}

\newtheorem*{Goal*}{Goal}

\newtheorem*{Assumption*}{Assumption}

\theoremstyle{remark}
\newtheorem*{remark*}{Remark}
\newtheorem{remark}[theorem]{Remark}

\numberwithin{equation}{section}

\newcommand{\ow}{\omega}
\newcommand{\Ow}{\Omega}
\newcommand{\lda}{\lambda}

\newcommand{\p}{\partial}

\newcommand{\J}{\mathcal{J}_{\widehat{\mathcal{R}}}}

\newcommand{\C}{{\mathbb{C}}}
\newcommand{\R}{{\mathbb{R}}}

\newcommand{\Z}{{\mathbb{Z}}}
\newcommand{\N}{{\mathbb{N}}}
\newcommand{\D}{{\mathbb{D}}}

\newcommand{\F}{{\mathbb{F}}}

\newcommand{\CP}{\C P}
\newcommand{\RP}{\R P}

\DeclareMathOperator{\codim}{codim}

\DeclareMathOperator{\id}{id}

\DeclareMathOperator{\Fix}{Fix}
\DeclareMathOperator{\pt}{pt}

\DeclareMathOperator{\HW}{HW}
\DeclareMathOperator{\CZ}{CZ}

\DeclareMathOperator{\SH}{SH}

\DeclareMathOperator{\Ho}{H}
\DeclareMathOperator{\st}{st}
\DeclareMathOperator{\Int}{Int}
\DeclareMathOperator{\FS}{FS}
\DeclareMathOperator{\ev}{ev}

\DeclareMathOperator{\univ}{univ}
\DeclareMathOperator{\reg}{reg}
\DeclareMathOperator{\Aut}{Aut}

\begin{document}

\title[On the topology of Lagrangian fillings]{On the topology of Lagrangian fillings of the standard Legendrian sphere}

\author{Joontae Kim and Myeonggi Kwon}
\address{Department of Mathematics, Sogang University, 35 Baekbeom-ro, Mapo-gu, Seoul 04107, Republic of Korea}
\email{joontae@sogang.ac.kr}
\address{Department of Mathematics Education, and Institute of Pure and Applied Mathematics, Jeonbuk National University, Jeonju 54896, Republic of Korea}
\email{mkwon@jbnu.ac.kr}

\begin{abstract}

In this paper we study the uniqueness of Lagrangian fillings of the standard Legendrian sphere $\mathcal{L}_0$ in the standard contact sphere $(S^{2n-1},\xi_{\st})$.
We show that every exact Maslov zero Lagrangian filling $L$ of $\mathcal{L}_0$ in a Liouville filling of $(S^{2n-1},\xi_{\st})$ is a homology ball.
If we restrict ourselves to real Lagrangian fillings, then $L$ is diffeomorphic to the $n$-ball for~$n\ge 6$.
\end{abstract}

\maketitle


\section{Introduction}\label{sec: intro}
Let $(W, d\lda)$ be a Liouville domain of dimension $2n$ with a Liouville form $\lda$.
Its boundary $\Sigma = \p W$ admits a natural contact structure $\xi = \ker \lda|_{\Sigma}$.
In this case, $(W,d\lambda)$ is called a \emph{Liouville filling} of a contact manifold $(\Sigma,\xi)$.
Given a Legendrian submanifold $\mathcal{L}\subset (\Sigma,\xi)$, we call $L\subset (W,d\lambda)$ an \emph{exact Lagrangian filling} of $\mathcal{L}$ if $L$ is a Lagrangian submanifold with $\p L=\mathcal{L}$, the 1-form $\lambda|_{TL}$ is exact, and the Liouville vector field of $\lambda$ is tangent to $TL$ near the boundary $\p L$.

The classification of Lagrangian fillings of Legendrian submanifolds has been an interesting problem in contact and symplectic topology.
For Legendrian $(2, n)$-torus links with maximal Thurston--Bennequin invariant, there are non-uniqueness results on exact Lagrangian fillings; \cite{Pan17} and \cite{STWZ19} give a lower bound of the number of exact Lagrangian fillings in terms of augmentations and cluster varieties.
Examples of Legendrian links with infinitely many distinct exact Lagrangian fillings are then found in \cite{CG20,CZ20,GSW20}.
In contrast,
it is shown in \cite{EP96} that the Legendrian unknot in the standard contact sphere $(S^3,\xi_{\st})$ with maximal Thurston--Bennequin invariant admits a unique exact Lagrangian filling in $\C^2$ up to compactly supported Hamiltonian isotopy. This result has been partially generalized to higher dimensions in \cite[Theorem~4.7 and Remark~4.3]{CDGG14} for the \emph{standard Legendrian sphere}
$$
\mathcal{L}_0 : = \{\mathbf{x}+i\mathbf{y} \in S^{2n-1} \subset \C^{n} \mid  \mathbf{y} = \mathbf{0}\} \cong S^{n-1}
$$
where the $n$-ball $L_0 = \{\mathbf{x}+i\mathbf{y}\in B^{2n} \mid \mathbf{y}=\mathbf{0}\} \cong B^{n}$ serves as an obvious Lagrangian filling. Using relations between bilinearized Legendrian contact homology and wrapped Floer homology, \cite{CDGG14} showed that every exact Lagrangian filling of $\mathcal{L}_0$ in the standard symplectic ball $(B^{2n}, d \lda_{\st})$ is contractible. 

In the first part of this paper, we investigate the uniqueness of Lagrangian fillings up to homology focusing on a special property of $\mathcal{L}_0$ called \emph{index-positivity}.
Roughly speaking, it is a relative version of the dynamical convexity for contact manifolds in symplectic field theory, which has played an important role in the uniqueness of symplectic fillings e.g. as in \cite{Laz20,Zh20}.
See Definition~\ref{def: indpos} for a precise definition.
The triple $(S^{2n-1}, \xi_{\st}, \mathcal{L}_0)$ of the standard Legendrian sphere $\mathcal{L}_0$ in $(S^{2n-1}, \xi_{\st})$ forms an example of index-positive triples, and this is a key ingredient of the following theorem. 



\begin{TheoremA}\label{thm: main1}
Let $(W, d\lambda)$ be a Liouville filling of $(S^{2n-1}, \xi_{\st})$ and $L \subset W$ an exact Lagrangian filling of $\mathcal{L}_0 \subset (S^{2n-1}, \xi_{\st})$ with vanishing Maslov class. Then $L$ has the same homology group $\Ho_*(L ; \Z)$ as the ball $B^n$.
\end{TheoremA}

We remark that a Liouville filling $(W, d\lda)$ in Theorem~\ref{thm: main1} is not necessarily the standard symplectic filling $(B^{2n},d\lambda_{\st})$, while $W$ is diffeomorphic to the ball $B^{2n}$ by Eliashberg--Floer--McDuff \cite{McD91}. The vanishing of the Maslov class $\mu_{L}\colon \pi_2(W, L) \rightarrow \Z$ is a technical condition for Floer theory, and this is the case for example when the first homology group $\Ho_1(L;\Z)$ is torsion.

We prove the above theorem using invariance properties of wrapped Floer homology. The wrapped Floer homology $\HW(L)$ of an admissible Lagrangian $L$, introduced in \cite{AboSei}, is in principle an invariant of  $L$, but sometimes its positive part $\HW^+(L)$ is completely determined by the boundary Legendrian.
This is for example the case when the Legendrian is index-positive.
Furthermore, the Legendrian boundary constraint $\mathcal{L}_0$ actually forces 
the positive wrapped Floer homology $\HW^+(L)$ of its Lagrangian filling $L$ to be isomorphic to the homology of the obvious filling $L_0 = B^n$.
This comes from the fact that any Liouville filling of $(S^{2n-1}, \xi_{\st})$ has vanishing symplectic homology \cite[Corollary~6.5]{Sei08}.
In conclusion, the homology of $L$ is completely governed by its boundary $\mathcal{L}_0$. 

Our argument applies to a more general setup with index-positive Legendrians and Lagrangian fillings with vanishing wrapped Floer homology; see Theorem \ref{thm: main1general}.



\begin{remark}\label{rem: introEFM} \
\begin{itemize}



\item There are infinitely many Legendrians in $(S^{2n-1}, \xi_{\st})$ with multiple exact Lagrangian fillings up to homeomorphism as in \cite[Theorem~1.5]{CGHS14}. Those Legendrians are not the standard Legendrian spheres by Theorem~\ref{thm: main1}. Also, examples of Legendrian submanifolds in high dimension with infinitely many exact Lagrangian fillings up to Hamiltonian isotopy can be found in \cite{Ro22}. 


\item As for another uniqueness result in high dimensions, let $\Sigma$ be an integral homology sphere with non-trivial $\pi_1$. In  \cite[Corollary~1.8]{AS12} it is shown 
that every exact Maslov zero Lagrangian filling of the Legendrian unknot in the cotangent bundle $T^*\Sigma$ is a homology ball. 
\end{itemize}
\end{remark}

\begin{remark}
In the setting of Theorem \ref{thm: main1}, it might be possible to prove further that every exact Lagrangian fillings of the standard Legendrian sphere is simply-connected and hence diffeomorphic to the ball in high dimensions. An idea is to utilize a twisted version of wrapped Floer homology following \cite{CDGG14}. In \cite[Section 4.2]{CDGG14}, a wrapped Floer homology $\underline{\HW}(L)$ of a Lagrangian filling $L$ with local coefficients via the universal covering $\widetilde L \rightarrow L$ is described; it is a vector space over the group ring  $\Z_2[\pi_1(L)]$ and is equipped with a canonical $\pi_1(L)$-action by deck transformations. If $\underline{\HW}(L)$ is vanishing as the non-twisted version, then one can deduce that there is a $\pi_1(L)$-equivariant isomorphism between the positive action part $\underline{\HW}^+(L)$ and the singular homology of $\widetilde L$ as in \cite[Corollary 4.8]{CDGG14}. By a similar algebraic argument to the proof of \cite[Proposition 4.9]{CDGG14}, this would imply that $\pi_1(L) = 0$. 

It is however not clear to the authors whether the twisted wrapped Floer homology $\underline{\HW}(L)$ is still vanishing, as in analogue to Ritter \cite[Theorem 10.6]{Rit} and Seidel--Smith \cite[Corollary 6.5]{Sei08}.
\end{remark}

In the second part of the paper, we restrict ourselves to a special class of exact Lagrangian fillings, called \emph{real Lagrangian fillings}. For those fillings, we can manipulate symmetry in $J$-holomorphic curves to show that they are diffeomorphic to the ball. 

Let $(\Sigma, \xi, \rho)$ be a \emph{real contact manifold}, meaning that $(\Sigma, \xi = \ker \alpha)$ is a contact manifold with a contact form~$\alpha$ and $\rho\colon \Sigma \rightarrow \Sigma$ is an anti-contact involution, i.e. $\rho^2=\id$ and $\rho^*\alpha = -\alpha$. 
A \emph{real Liouville filling} $(W, \ow, \mathcal{R})$ of $(\Sigma, \xi = \ker \alpha, \rho)$ is a Liouville filling $(W,d\lambda)$ equipped with an anti-symplectic involution $\mathcal{R}\colon W \rightarrow W$ such that $\mathcal{R}$ is exact, i.e. $\mathcal{R}^*\lambda=-\lambda$, and $\mathcal{R}|_{\p W}=\rho$.
If the fixed point sets $\Fix(\mathcal{R})$ and $\Fix(\rho)$ are non-empty (and hence they are Lagrangian and Legendrian, respectively), then we say $\Fix(\mathcal{R})$ is a \emph{real Lagrangian filling} of the \emph{real Legendrian} $\Fix(\rho)$.
Since $\rho$ is exact, any real Lagrangian filling is exact.
The Liouville domain $(B^{2n},d\lambda_{\st})$ endowed with complex conjugation $\mathcal{R}_0(x,y)=(x,-y)$ is a real Liouville filling of the standard real contact sphere $(S^{2n-1},\xi_{\st},\rho_0:=\mathcal{R}|_{S^{2n-1}})$.
Moreover, $L_0\subset (B^{2n},d\lambda_{\st},\mathcal{R}_0)$ serves a canonical real Lagrangian filling of $\mathcal{L}_0$.
The following result shows that the real Legendrian boundary constraint determines the diffeomorphism type
of real Lagrangian fillings uniquely in high dimensions.

\begin{TheoremA}\label{thm: main2}
If $n\ge 6$, then every real Lagrangian filling of $\mathcal{L}_0$ in a real Liouville filling of $(S^{2n-1},\xi_{\st},\rho_0)$ is diffeomorphic to the ball $B^n$.

\end{TheoremA}
We refer to Theorem~\ref{thm: main2-1} for a more general statement.

\begin{remark}\label{rem: smith}\
\begin{itemize}
\item The classical Smith inequality (i.e. $\dim \Ho_*(W;\Z_2)\ge \dim \Ho_*(L;\Z_2)$) already tells us that any Lagrangian filling $L=\Fix(\mathcal{R})\subset (W,d\lambda,\mathcal{R})$ in Theorem~\ref{thm: main2} is a $\Z_2$-homology ball since $W$ is diffeomorphic to $B^{2n}$.

\item Theorem~\ref{thm: main2} implies that the anti-contact involution on $(S^{2n-1},\xi_{\st})$ given by complex conjugation can only be extended to an exact anti-symplectic involution on a Liouville filling of $(S^{2n-1},\xi_{\st})$ in such a way that the corresponding real Lagrangian is diffeomorphic to $B^n$.
\end{itemize}
\end{remark}
The proof of Theorem~\ref{thm: main2} employs the filling-by-holomorphic-curves technique, which goes back to \cite{El90}.
A key observation is that the real contact manifold $(S^{2n-1}, \xi_{\st}, \rho_0)$ admits a \emph{real} symplectic capping, described in Section~\ref{sec: cappingconstruction}, which is foliated by $J$-holomorphic disks with boundary on a real Lagrangian.
This allows us to perform a real version of the degree method, as in \cite{BarGeiZeh19, McD91}, via $J$-holomorphic disks. 
The main technical point lies in the fact that we deal with $\Z_2$-anti-invariant almost complex structures.
This allows us to make use of analysis of $J$-holomorphic \emph{spheres}, well-studied in \cite{BarGeiZeh19, McD91}, to understand the behavior of $J$-holomorphic disks with boundary on a real Lagrangian; those disks are the half of spheres.
A price to pay is to achieve an equivariant transversality with respect to $\Z
_2$-symmetries by anti-symplectic involutions.
This can be done, as observed in \cite{Kim21}, by showing that a naturally induced involution on the moduli space of $J$-holomorphic disks has no fixed points, see Lemma~\ref{lem: simple_nofixedpt}.

\subsection*{Organization of the paper} In Section~\ref{sec: homologygroupofLagrangianfillings} we first review wrapped Floer homology and discuss the key invariance property under index-positivity conditions described in Section~\ref{sec: independenceofHW+}.
The proof of Theorem~\ref{thm: main1} is then given in Section~\ref{sec: proofofthm1.1}.
For the proof of the result on real Lagrangian fillings, we introduce the capping construction in Section~\ref{sec: cappingconstruction} and study $J$-holomorphic disks in Section~\ref{sec: moduli space} with a focus on equivariant transversality.
The main step of the degree method for the proof of Theorem~\ref{thm: main2} is done in Section~\ref{sec: evalutation}.

\section{Homology group of Lagrangian fillings} \label{sec: homologygroupofLagrangianfillings}

\subsection{Wrapped Floer homology} \label{sec: wrapped Floer homology}
We briefly explain basic notions in wrapped Floer homology.
We refer to \cite{AboSei} for more details. 

Let $(W^{2n}, d\lda)$ be a Liouville domain with a Liouville form $\lda$.
The boundary $\p W$ admits a contact structure $\xi = \ker \alpha$, where $\alpha = \lda|_{\p W}$.
A Lagrangian $L \subset W$ is called \emph{admissible} if it is exact (i.e. $\lda|_{TL}$ is exact), the Liouville vector field on $(W, d\lda)$ is tangent to $TL$ near the boundary, and $L$ intersects the boundary $\p W$ in a Legendrian $\p L \subset \p W$.
In particular, exact Lagrangian fillings are admissible.
For an admissible Lagrangian $L$ in a Liouville domain $(W, d\lda)$ we can complete a pair $(W,L)$ by attaching part of the symplectization $(\R_{\geq 1} \times \p W, d(r \alpha))$, where $r \in \R$, along the boundary $\p W$ via the Liouville flow:
$$
\widehat W  = W \cup_{\p W} (\R_{\geq 1} \times \p W), \quad \widehat L = L \cup_{\p L} (\R_{\geq 1} \times \p L).
$$
The \emph{wrapped Floer homology} $\HW(L)$ of $L$ is the Lagrangian Floer homology of $\widehat L$ using a quadratic Hamiltonian on $\widehat{W}$ whose chain complex is generated by contractible Hamiltonian 1-chords relative to $\widehat L$.
It consists of two classes of generators corresponding to Morse critical points on $L$ and to Reeb chords in $(\p W, \alpha, \p L)$.
The boundary map is defined by counting the Floer strips.

\begin{remark}\label{rem: topcondHW}
To equip $\HW(L)$ with $\Z$-grading, it is common to assume that the Maslov class $\mu_L\colon \pi_2(W, L) \rightarrow \Z$ vanishes as in Theorem~\ref{thm: main1}. 
If $L$ is spin, i.e. the second Stiefel--Whitney class of $L$ vanishes, then we can define $\HW(L)$ over any field $\F$, see \cite[Section~9.1]{AboSei}.
Otherwise, we define $\HW(L)$ with $\F = \Z_2$.
\end{remark}

\begin{example}\label{ex: ballexample}
The simplest example of Liouville domains is the closed unit ball $B^{2n}\subset \C^{n}$ endowed with the standard Liouville form $\lda_{\st} = \frac{1}{2}\sum_j (x_jdy_j-y_jdx_j)$, where $\mathbf{x}+i\mathbf{y}\in B^{2n}$.
The restricted  $1$-form $\alpha_{\st} = \lda_{\st}|_{S^{2n-1}}$ to the boundary defines the \emph{standard contact form} on~$S^{2n-1}$ and its kernel $\xi_{\st}=\ker \alpha_{\st}$ is the \emph{standard contact structure}.
The $n$-ball $L_0=\{\mathbf{x}+i\mathbf{y} \in B^{2n}\mid \mathbf{y} = \mathbf{0}\}$ is an admissible Lagrangian, and its Legendrian boundary is the standard Legendrian sphere $\mathcal{L}_0 \subset S^{2n-1}$.
Notice that the topological conditions in Remark~\ref{rem: topcondHW} are obviously satisfied.
\end{example}

The \emph{positive wrapped Floer homology} $\HW^+_*(L; \F)$ is defined by filtering out the generators corresponding to Morse critical points on $L$ from the full chain complex. A useful algebraic property is the following tautological exact sequence \cite[Lemma~8.3]{Rit}:
\begin{equation}\label{eq: taules}
\rightarrow \Ho_{k+n}(L, \p L; \F) \rightarrow \HW_k(L; \F) \rightarrow \HW_k^{+}(L; \F) \rightarrow \Ho_{k+n-1}(L, \p L; \F) \rightarrow
\end{equation}
\subsection{Independence of $\HW^+$} \label{sec: independenceofHW+} In this section we recall a well-known invariance property of $\HW^+(L;\F)$ under an \emph{index-positivity} condition on a Legendrian boundary $\p L$. 
\begin{definition}\label{def: indpos}
A triple $(\Sigma, \xi, \mathcal{L})$ consisting of a contact manifold $(\Sigma^{2n-1}, \xi)$ and a Legendrian $\mathcal{L} \subset \Sigma$ is called \emph{index-positive} if $c_1(\xi)|_{\pi_2(\Sigma)} = 0$, the Maslov class $\mu_\mathcal{L}\colon \pi_2(\Sigma,\mathcal{L})\to \Z$ of $\mathcal{L}$ vanishes, and there exists a contact form $\alpha$ for $(\Sigma, \xi)$ with the following properties on the Conley--Zehnder indices:
\begin{itemize}
\item Every contractible periodic Reeb orbit $\gamma$ in $(\Sigma, \alpha)$ is non-degenerate and $\CZ(\gamma) > 3-n$;
\item Every Reeb chord $x$ of $(\Sigma, \alpha, \mathcal{L})$ that is trivial in $\pi_1(\Sigma,\mathcal{L})$ is non-degenerate and $\CZ(x) > 0$.
\end{itemize}



\end{definition}

\begin{remark}
Here the Conley--Zehnder indices of periodic Reeb orbits and Reeb chords are the ones in \cite[Section 9.5]{CiOa18}.
The latter also matches the degree of Reeb chords given in \cite[Equation (2.1)]{Ek12}.
\end{remark}

The following invariance property is extracted from \cite[p.\ 2105, Section~9.5]{CiOa18}. See also \cite[Corollary 2.12]{BaKw21}.
The condition $\pi_1(\Sigma, \mathcal{L}) = 0$ below insures that contractible chords in $(W, L)$ are contractible in $(\Sigma, \mathcal{L})$; cf. \cite[p.\ 2105, Condition (ii)]{CiOa18}.

\begin{proposition}\label{prop: invHWplus}
Let $(\Sigma, \xi,\mathcal{L})$ be an index-positive triple with $\pi_1(\Sigma, \mathcal{L}) = 0$. If there are two exact Maslov zero Lagrangian fillings $L$ and $L'$ of $\mathcal{L}$ (in Liouville fillings $W$ and $W'$ of $(\Sigma,\xi)$, respectively) which are spin, then $\HW^+(L; \F) \cong \HW^+(L' ; \F)$ as groups.


\end{proposition}

We can apply the invariance property to the case of the standard Legendrian sphere $\mathcal{L}_0$ and its Lagrangian fillings:

\begin{example}\label{ex: invofHW+}
We claim that the triple $(S^{2n-1},\xi_{\st}, \mathcal{L}_0)$ is index-positive.
This can be seen from a direct computation with respect to a perturbed non-degenerate contact form on $(S^{2n-1}, \xi_{\st})$.
Using the standard complex coordinates, the contact form is written by
$$
\tilde{\alpha}  = \frac{i}{2} \sum_{j=0}^{n} a_j (z_j d\overline z_j - \overline z_j dz_j)|_{S^{2n-1}},
$$ 
where the coefficients $a_0, a_1, \dots, a_n\in \R_{>0}$ are rationally independent. The Reeb flow is given by coordinate-wise rotations on $\C^{n}$
$$
(z_0, z_1, \dots, z_n) \longmapsto (e^{it/a_0}z_0, e^{it/a_1}z_1, \dots, e^{it/a_n}z_n) 
$$
and periodic Reeb orbits are of the form
$$
\gamma_j^m(t) = (0, \dots, 0, e^{it/a_j}z_j, 0, \dots, 0),
$$
for $m\in \N$, $1\le j\le n$, and $z_j \in S^1$, with period $2\pi m a_j$. Note also that Reeb chords in $(S^{2n-1}, \tilde{\alpha}, \mathcal{L}_0)$ are of the form
$$
x_j^m(t) = (0, \dots, 0, \pm e^{it/a_j}, 0, \dots, 0),
$$
with period $\pi m a_j$.
We can compute their indices explicitly:
\begin{align*}
\CZ(\gamma_j^m) &= n-1 + 2\sum_{i \neq j} \left\lfloor \frac{ma_j}{a_i} \right \rfloor + 2m,\\ 
\CZ(x_j^m) &= n-2 + \sum_{i \neq j} \left\lfloor \frac{ma_j}{a_i} \right \rfloor + m.
\end{align*}
See for example \cite[Lemma~2.1]{GH18}.
Assuming $a_1 < a_2, \dots, a_n$, we see that $\gamma_1^1$ and $x_1^1$ attain the minimum  indices with values
$$
\CZ(\gamma_1^1) = n +1 , \quad \CZ(x_1^1) = n-1.
$$ 
It follows that the triple $(S^{2n-1},\xi_{\st}, \mathcal{L}_0)$ is index-positive for $n >1$.
\end{example}
Therefore, every exact Maslov zero Lagrangian filling $L$ of $\mathcal{L}_0$ has the same positive wrapped Floer homology group $\HW^+(L;\Z_2)$ by Proposition~\ref{prop: invHWplus}.

\subsection{Proof of Theorem~\ref{thm: main1}}  \label{sec: proofofthm1.1}

\subsubsection{Over $\Z_2$-coefficients} We first compute the homology group $\Ho_*(L, \p L; \Z_2)$. Over $\Z_2$-coefficients, the wrapped Floer homology group of $L$ is well-defined without assuming $L$ is spin. By \cite[Corollary~6.5]{Sei08}, the symplectic homology $\SH_*(W)$ is vanishing. It follows that $\HW_*(L; \Z_2)$ also vanishes since it admits a module structure over the ring $\SH_*(W; \Z_2)$ as in \cite[Theorem~6.17]{Rit}. In view of the tautological exact sequence \eqref{eq: taules}
$$
\rightarrow \Ho_{k+n}(L, \p L; \Z_2) \rightarrow \HW_k(L; \Z_2) \rightarrow \HW_k^{+}(L; \Z_2) \rightarrow \Ho_{k+n-1}(L, \p L; \Z_2) \rightarrow
$$
we deduce that
\begin{equation}\label{eq: positivesingularisom}
\HW^{+}_k(L; \Z_2) \cong \Ho_{k+n-1}(L, \p L; \Z_2)
\end{equation}
for every $k \in \Z$. 
In particular, the standard Lagrangian filling $L_0$ for $\mathcal{L}_0$ in Example~\ref{ex: ballexample} satisfies
$$
\HW^{+}_k(L_0; \Z_2) \cong \Ho_{k+n-1}(B^n, S^{n-1}; \Z_2) \cong \begin{cases} \Z_2 & \text{for $k = 1$}; \\ 0 & \text{otherwise.}   \end{cases}
$$
Applying the invariance property of $\HW^+$ as in Proposition~\ref{prop: invHWplus}, we obtain
$$
 \HW^{+}_k(L; \Z_2) \cong \HW^+_k(L_0;\Z_2),
$$
and hence
\begin{equation}\label{eq: relhomcompu}
\Ho_{k}(L, \p L; \Z_2) \cong \begin{cases} \Z_2 & \text{for $k = n$}; \\ 0 & \text{otherwise.}   \end{cases}
\end{equation}
In particular, the cohomology groups $\Ho^1(L;\Z_2)$ and $\Ho^2(L; \Z_2)$ are vanishing so that $L$ is orientable and spin.


\subsubsection{Over $\Z$-coefficients} To prove Theorem~\ref{thm: main1}, it is enough to show that the reduced homology $\tilde \Ho_*(L; \F) =0$ for any field $\F$, see \cite[Corollary 3A.7]{Hat02}.
Since $L$ is spin as shown in the previous section, the wrapped Floer homology group of $L$ is now well-defined over $\F$. Provided this, notice that the argument to compute $\Ho_{*}(L, \p L; \Z_2)$ in the previous section still works over $\F$ in exactly the same way. This yields
$$
\Ho_{k}(L, \p L; \F) \cong \begin{cases} \F & \text{for $k = n$}; \\ 0 & \text{otherwise.}   \end{cases}
$$
In view of the long exact sequence for the pair $(L, \p L)$
\begin{equation}\label{eq: lesforpair}
\rightarrow \Ho_k(\p L; \F) \rightarrow \Ho_k(L; \F)  \rightarrow \Ho_k(L, \p L; \F) \rightarrow \Ho_{k-1}(\p L; \F) \rightarrow
\end{equation}
it is direct to see that $\tilde \Ho_k(L; \F) = 0$ for $k \leq n-2$. In addition, since $L$ is a smooth manifold with boundary and is orientable as in the previous section, we have $\Ho_{n}(L; \F)  = 0$. Considering the top-dimensional part of the long exact sequence \eqref{eq: lesforpair}
$$
0  \rightarrow \Ho_n(L, \p L; \F) \xrightarrow{\p_*} \Ho_{n-1}(\p L; \F) \rightarrow \Ho_{n-1}(L; \F) \rightarrow 0
$$
it follows that the connecting map $\p_*\colon \Ho_n(L, \p L; \F) \cong \F \rightarrow \Ho_{n-1}(\p L; \F) \cong \F$ is an isomorphism. This implies $\Ho_{n-1}(L;\F) = 0$ by the exactness. Summing up the computations, we conclude that $\tilde \Ho_*(L;\F) = 0$. This finishes the proof of Theorem~\ref{thm: main1}.

The argument so far actually proves the following more general statement:

\begin{theorem}\label{thm: main1general}
Let $(\Sigma, \xi, \mathcal{L})$ be an index-positive triple with $\pi_1(\Sigma, \mathcal{L}) = 0$.
If there are two exact spin Maslov zero Lagrangian fillings $L$ and $L'$ of $\mathcal{L}$ (in Liouville fillings $W$ and $W'$ of $(\Sigma,\xi)$, respectively) having vanishing wrapped Floer homology, then $\Ho^*(L)$ and $\Ho^*(L')$ are isomorphic.
\end{theorem}

\begin{remark}\
\begin{itemize}
\item Using the fact that $\HW_*(L)$ vanishes if and only if its cohomology version $\HW^*(L)$ vanishes (see \cite[Theorem~10.4]{Rit} and \cite[Theorem~3.4]{CiOa18} for version in symplectic homology), one can show that $\Ho_*(L)\cong \Ho_*(L')$ in the situation of Theorem~\ref{thm: main1general}.
\item One can generalize Theorem~\ref{thm: main1general} in terms of \emph{asymptotically dynamically convexity} for Legendrians, introduced in \cite{Laz20}. 

\end{itemize}
\end{remark}

\section{Diffeomorphism type of real Lagrangian fillings} \label{sec: Diffeomorphismtype}

\subsection{Real symplectic fillings}\label{sec: real symplectic fillings}


We introduce a generalized notion of real Liouville fillings, which is enough for our purposes.
A symplectic manifold $(W, \ow)$ equipped with an anti-symplectic involution $\mathcal{R}$ is called a \emph{real symplectic manifold}.

\begin{definition}\label{def: realsymplecticfilling}
A \emph{real symplectic filling} $(W, \ow, \mathcal{R})$ of a real contact manifold $(\Sigma, \alpha, \rho)$ is a real symplectic manifold $(W,\ow,\mathcal{R})$ such that
\begin{enumerate}
\item $(W, \ow)$ is a strong symplectic filling of the contact manifold $(\Sigma, \xi)$ as in \cite[Definition~5.1.1]{Gei08} so that there is a Liouville form $\lda$ for $\ow$ near the boundary $\p W = \Sigma$ whose Liouville vector field is pointing outwards along $\p W$ and $\lda|_{\Sigma} = \alpha$;
\item $\mathcal{R}$ is an anti-symplectic involution which is exact near the boundary, i.e. $\mathcal{R}^* \lda = -\lda$ near $\p W$.
\end{enumerate}
\end{definition}


The goal of this section is to prove the following theorem, which is more general than Theorem~\ref{thm: main2}.

\begin{theorem}\label{thm: main2-1}
Let $(W, \ow, \mathcal{R})$ be a real symplectic filling of the standard real contact sphere $(S^{2n-1}, \xi_{\st}, \rho_0)$ that is symplectic aspherical, i.e. $\ow$ vanishes on $\pi_2(W)$.
If $n \geq 4$, then the real Lagrangian filling $L = \Fix(\mathcal{R})$ of $\mathcal{L}_0$ satisfies $\pi_1(L) =0$ and $\tilde \Ho_k(L;\Z) = 0$ for all $k\ge 0$.
In particular, $L$ is diffeomorphic to the ball $B^n$ for $n\ge 6$.
\end{theorem}

\subsection{Capping construction} \label{sec: cappingconstruction}
Let $(W,\ow,\mathcal{R})$ be a real symplectic filling of the standard real contact sphere $(S^{2n-1},\xi_{\st},\rho_0)$.
We follow the construction in \cite[Section~2.1]{BarGeiZeh19} (see also \cite[Section~3.2]{McD91}) with additional care about real structures.
Since $\mathcal{R}^* \lda = -\lda$ near the boundary, a collar neighborhood of $\p W\subset W$ can be identified with
$$
(\nu(\p W), \ow, \mathcal{R}) \cong ((1-\epsilon, 1] \times S^{n-1}, d(r\alpha_{\st}), \id \times \rho_0).
$$
This allows us to construct a real symplectic manifold replacing the ball $B^{2n}$ in $\C^{n}$ by $W$ as follows:
$$
(Z, \Ow, \mathcal{R}') : = (W, \ow, \mathcal{R}) \cup_{(S^{2n-1}, \xi_{\st}, \rho_0)} (\C^n \setminus \Int B^{2n}, \ow_0, \mathcal{R}_0).
$$
Here $\ow_0=d\lambda_{\st}$ is the standard symplectic form on $\C^n$ and $\mathcal{R}_0$ is complex conjugation. By compactifying the last $\C$-factor of $\C^{n} = \C^{n-1} \times \C$ to be $\C \cup \{\infty\}  = \CP^1$, we obtain the real symplectic manifold
\begin{equation}\label{eq: construction_zhat}
(\widehat Z, \widehat \Ow, \widehat{\mathcal{R}}) : =  (W, \ow, \mathcal{R}) \cup_{(S^{2n-1}, \xi_{\st},\rho_0)} (\C^{n-1} \times \CP^1 \setminus \Int B^{2n}, \ow_0 \oplus \ow_{\FS}, \mathcal{R}_0 \times \rho_{\CP^1}),
\end{equation}
where $\ow_{\FS}$ denotes the Fubini--Study form and $\rho_{\CP^1}$ is complex conjugation on $\CP^1$. Denote $L: = \Fix(\mathcal{R}) \subset W$.
The real Lagrangian of $(\widehat Z, \widehat \Ow, \widehat{\mathcal{R}})$ is then given by
$$
\widehat L : = \Fix(\widehat{\mathcal{R}}) = L \cup_{\mathcal{L}_0}  (\R^{n-1} \times \RP^1 \setminus \Int B^n).
$$
See Figure~\ref{fig: zhat}.

\begin{figure}[h]
\begin{tikzpicture}[scale=0.85]

\begin{scope}[xshift=8cm]
\draw (-1,2)--(-1,0.5);
\draw (-1,-2)--(-1,-0.5);
\draw [dashed] (-1,-0.5)--(-1,0.5);
\draw (2,-2)--(2,2);
	
\draw [dashed] (2,-2) arc [x radius = 1.5cm, y radius = 0.2cm, start angle = 0,
  end angle = 180];
\draw (2,-2) arc [x radius = 1.5cm, y radius = 0.2cm, start angle = 0,
  end angle = -180];
\draw (2,2) arc [x radius = 1.5cm, y radius = 0.2cm, start angle = 0,
  end angle = 180];
\draw (2,2) arc [x radius = 1.5cm, y radius = 0.2cm, start angle = 0,
  end angle = -180];
  
\draw (0.2,-0.7) arc [x radius = 0.7cm, y radius = 0.7cm, start angle = -90,
  end angle = 90];
\draw [dashed](0.2,0.7) arc [x radius = 0.7cm, y radius = 0.7cm, start angle = 90,
  end angle = 270];

\draw plot [smooth] coordinates {(0.2,0.7) (0.05, 0.7) (-1,0.5) (-2.5,1.5) (-3.2, 1.5) (-3.9, 0.9) (-4.2,0.2) (-3.9,-0.9) (-3.2, -1.5) (-2.5,-1.5) (-1,-0.5) (0.05, -0.7) (0.2,-0.7)};

\node at (-4.5,1.4) {$(W,\ow)$};
\node at (-1.4,1.5) {$\widehat{Z}$};
\node at (1.4, 0.5) {$S^{2n-1}$};
\node [red] at (-0.2, -1.5) {$\widehat{L}$};

\draw[->] (2.7,-2.5)--(2.7, 2.3);
\node at (3.4, 1.5) {$\C^{n-1}$};
\draw[->] (2.7, -2.5) arc [x radius = 2cm, y radius = 0.3cm, start angle = 0,
  end angle = -150];
\node at (-1.6, -2.6) {$\C P^{1}$};

\draw [red, thick] (1.5, 2.15)--(1.5, 1.85);
\draw [red, dashed, thick] (1.5, 1.8)--(1.5, -1.85);
\draw [red, thick] (0.2, -2.2)--(0.2, -0.7);
\draw [red, thick] (0.2, 1.8)--(0.2, 0.7);
\draw [thick, red] plot [smooth] coordinates {(0.2,0.7) (0.05, 0.5) (-1,0.2) (-2.5,1) (-3.2, 1) (-3.7, 0.5) (-3.8,0.2) (-3.8,-0.2) (-3.7, -0.5) (-3.2, -1) (-2.5, -1) (-1,-0.2) (0.05, -0.5) (0.2,-0.7)};

\draw [thick, blue] (0.2, 1.7) arc [x radius = 1.2cm, y radius = 0.2cm, start angle = -120,
  end angle = 0];
\draw [dashed, thick, blue] (2, 1.9) arc [x radius = 1.2cm, y radius = 0.2cm, start angle = 0,
  end angle = 55];
  
\begin{scope}[yshift=-0.2cm]
\draw [thick, blue] (0.2, 1.7) arc [x radius = 1.2cm, y radius = 0.2cm, start angle = -120,
  end angle = 0];
\draw [dashed, thick, blue] (2, 1.9) arc [x radius = 1.2cm, y radius = 0.2cm, start angle = 0,
  end angle = 55];
\end{scope}

\begin{scope}[yshift=-0.4cm]
\draw [thick, blue] (0.2, 1.7) arc [x radius = 1.2cm, y radius = 0.2cm, start angle = -120,
  end angle = 0];
\draw [dashed, thick, blue] (2, 1.9) arc [x radius = 1.2cm, y radius = 0.2cm, start angle = 0,
  end angle = 55];
\end{scope}  
  
\begin{scope}[yshift=-0.6cm]
\draw [thick, blue] (0.2, 1.7) arc [x radius = 1.2cm, y radius = 0.2cm, start angle = -120,
  end angle = 0];
\draw [dashed, thick, blue] (2, 1.9) arc [x radius = 1.2cm, y radius = 0.2cm, start angle = 0,
  end angle = 55];
\end{scope}

\begin{scope}[yshift=-3.7cm]
\draw [thick, blue] (0.2, 1.7) arc [x radius = 1.2cm, y radius = 0.2cm, start angle = -120,
  end angle = 0];
\draw [dashed, thick, blue] (2, 1.9) arc [x radius = 1.2cm, y radius = 0.2cm, start angle = 0,
  end angle = 55];
\end{scope}

\begin{scope}[yshift=-3.5cm]
\draw [thick, blue] (0.2, 1.7) arc [x radius = 1.2cm, y radius = 0.2cm, start angle = -120,
  end angle = 0];
\draw [dashed, thick, blue] (2, 1.9) arc [x radius = 1.2cm, y radius = 0.2cm, start angle = 0,
  end angle = 55];
\end{scope}

\begin{scope}[yshift=-3.3cm]
\draw [thick, blue] (0.2, 1.7) arc [x radius = 1.2cm, y radius = 0.2cm, start angle = -120,
  end angle = 0];
\draw [dashed, thick, blue] (2, 1.9) arc [x radius = 1.2cm, y radius = 0.2cm, start angle = 0,
  end angle = 55];
\end{scope}

\begin{scope}[yshift=-3.1cm]
\draw [thick, blue] (0.2, 1.7) arc [x radius = 1.2cm, y radius = 0.2cm, start angle = -120,
  end angle = 0];
\draw [dashed, thick, blue] (2, 1.9) arc [x radius = 1.2cm, y radius = 0.2cm, start angle = 0,
  end angle = 55];
\end{scope}

\end{scope}
\end{tikzpicture}
\caption{The pair $(\widehat{Z},\widehat{L})$ and standard holomorphic disks}
\label{fig: zhat}
\end{figure}
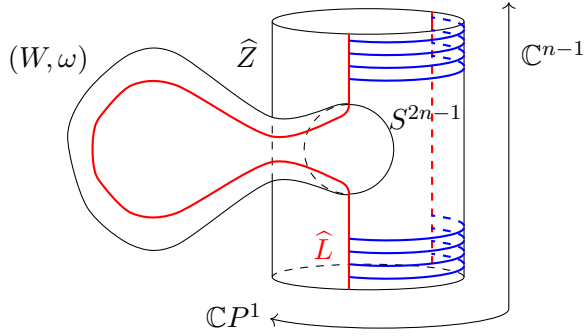
\noindent
Later, it will be useful to consider the following complex hypersurface
\begin{equation}\label{eq: cpxhyp}
H : = \C^{n-1} \times \{\infty\} \subset \widehat Z
\end{equation}
with the intersection number $[\pt \times \CP^1] \cdot [H] = 1$. Observe that $H$ is invariant under the involution $\widehat{\mathcal{R}}$.

\begin{remark}\label{rem: sameospherecase}
Note that $Z$ and $\widehat Z$ are basically the same ambient symplectic manifolds as those in \cite{BarGeiZeh19} and \cite{McD91} where the behavior of $J$-holomorphic sphere is studied. Our proof of Theorem~\ref{thm: main2} will highly rely on this.
\end{remark}

\subsection{The moduli space}  \label{sec: moduli space}
 
Denote the standard complex structure on $\C^{n-1}$ and $\CP^1$ by $J_0$ and $j$, respectively.
Let $\mathcal{J}_{\widehat{\mathcal{R}}}$ be the space of $\widehat \Ow$-compatible almost complex structures $J$ on $\widehat Z$ such that
\begin{itemize}
\item $J$ is of the form $J_0 \oplus j$ on $\widehat Z \setminus \Int W$;
\item $J$ is $\widehat{\mathcal{R}}$-anti-invariant, i.e. $\widehat{\mathcal{R}}^*J =- J$.
\end{itemize}
Take $J\in \J$. Let $D^2_+$ be the disk in $\CP^1 = \C \cup \{\infty\}$ given by the compactification of the upper half plane (including the real line $\R$) in $\C$.
Denote the other disk in $\CP^1$ by $D_-^2$ so that $\CP^1 = D^2_+ \cup D^2_-$ and $\R P^1=D_+^2\cap D_-^2$.
We abbreviate by $D^2\subset \C$ the closed unit disk.
For $v_0 \in \R^{n-1}$ with the norm $\|v_0\| > 1$ and a biholomorphism $\tau$ from $D^2$ to $D_+^2$, we have a $J$-holomorphic disk $u \colon (D^2, \p D^2) \rightarrow (\widehat Z, \widehat L)$ of the form $
u(z) = (v_0, \tau(z))$, and its homology class is given by $[u]  = [\pt \times D^2_+] \in \Ho_2(\widehat Z, \widehat L)$.
We call it a \emph{standard holomorphic disk}.
The region in $\widehat L$ consisting of $(v, z)\in \R^{n-1}\times D^2_+$ with $\|v\|>1$ is foliated by the boundary of those standard ones, see Figure~\ref{fig: zhat}.
Pick pairwise distinct boundary points $w_1, w_2, w_\infty \in \p D^2$ and pairwise distinct points $z_1,z_2,z_\infty=\infty\in \R P^1$ away from the glued region in \eqref{eq: construction_zhat}.
We define the moduli space $\mathcal{M}$ of $J$-holomorphic disks $u\colon (D^2, \p D^2) \rightarrow (\widehat Z, \widehat L)$ such that 
\begin{itemize}
\item its homology class $[u]$ represents $[\pt \times D^2_+] \in \Ho_2(\widehat Z, \widehat L)$;
\item $u(w_i) \in \R^{n-1} \times \{z_i\}$ for $i=1,2,\infty$.
\end{itemize}

\begin{proposition} \label{prop: eqregular}
Let $n\ge 4$. For generic $J\in \J$, the moduli space $\mathcal{M}$ is an oriented smooth manifold of dimension $n-1$.
\end{proposition}
The rest of this section is devoted to give a proof of Proposition~\ref{prop: eqregular}. A technical point here is that $J$ is $\widehat{\mathcal{R}}$-anti-invariant, and this requires to achieve an equivariant transversality.

We first introduce the notion of the double of holomorphic disks with boundary on a real Lagrangian.
Recall that $\rho_{\C P^1}$ is the anti-holomorphic involution of $\C P^1$ whose fixed point set is~$\R P^1$ and that $\tau$ is a biholomorphism from $D^2$ to $D_+^2$.
For a $J$-holomorphic disk $u\colon (D^2, \p D^2) \rightarrow (\widehat Z, \widehat L)$, 
its \emph{double} $u^\# \colon \CP^1 \rightarrow \widehat Z$ is a smooth $J$-holomorphic sphere defined~by
$$
u^\# (z) = \begin{cases} u(\tau^{-1}(z)) & \text{for $z \in D^2_+$}, \\ \widehat{\mathcal{R}}\big(u(\tau^{-1}(\rho_{\C P^1}(z)))\big) & \text{for $z \in D^2_-$}, \end{cases}
$$
see \cite[Section~5.2]{Kim21}.
The homology class $[u^\#] \in \Ho_2(\widehat Z)$ satisfies
$$
j_*(u^\#) = [u] - \widehat{\mathcal{R}}_*[u],
$$
where $j_*\colon \Ho_2(\widehat Z) \rightarrow \Ho_2(\widehat Z, \widehat L)$ is the natural map in the long exact sequence for the pair~$(\widehat{Z},\widehat{L})$.

The following lemma will be crucial later.
\begin{lemma}\label{lem: notspherical}Let $n\ge 4$ and $J\in \J$.
\begin{enumerate}
	\item The natural map $j_*\colon \Ho_2(\widehat Z) \rightarrow \Ho_2(\widehat Z, \widehat L)$ is injective.
	\item The homology class $[\pt \times D^2_+] \in \Ho_2(\widehat Z, \widehat L)$ is not $J$-spherical, i.e. it is not the image of the homology class of a $J$-holomorphic sphere under the map $j_*\colon \Ho_2(\widehat Z) \rightarrow \Ho_2(\widehat Z, \widehat L)$.  
\end{enumerate}
\end{lemma}

\begin{proof}
We first prove that the map $j_*\colon \Ho_2(\widehat Z) \rightarrow \Ho_2(\widehat Z, \widehat L)$ is injective. In view of the long exact sequence for the pair $(\widehat Z, \widehat L)$
$$
\rightarrow \Ho_2(\widehat L) \xrightarrow{i_*} \Ho_2(\widehat Z) \xrightarrow{j_*} \Ho_2(\widehat Z, \widehat L) \rightarrow
$$
it suffices to show that the induced map $i_* \colon \Ho_2(\widehat L) \rightarrow \Ho_2(\widehat Z)$ of the inclusion  is trivial.
Considering the Mayer--Vietoris sequences for the decompositions $$\widehat L = L \cup_{\mathcal{L}_0}  (\R^{n-1} \times \RP^1 \setminus \Int B^n), \quad \widehat Z = W \cup_{S^{2n-1}}(\C^{n-1} \times \CP^1 \setminus \Int B^{2n})$$ we obtain the following commutative diagram:
$$
\begin{tikzcd}
\Ho_2(\mathcal{L}_0) \arrow{r}  \arrow{d} &  \Ho_2(L) \oplus \Ho_2(\R^{n-1} \times \RP^1 \setminus \Int B^n) \arrow{r}{\Phi_*} \arrow[swap]{d}{i_* \oplus i_*} & \Ho_2(\widehat L)  \arrow{d}{i_*} \arrow{r} & \Ho_1(\mathcal{L}_0)  \arrow{d} \\%
\Ho_2(S^{2n-1}) \arrow{r} & \Ho_2(W) \oplus \Ho_2(\C^{n-1} \times \CP^1 \setminus \Int B^{2n}) \arrow{r}{\Psi_*}& \Ho_2(\widehat Z) \arrow{r} & \Ho_1(S^{2n-1})
\end{tikzcd}
$$
Here the horizontal arrows are of the Mayer--Vietoris sequences and the vertical arrows come from the exact sequences for the  respective pairs.
Since $n \geq 4$, the groups $\Ho_k(\mathcal{L}_0 = S^{n-1})$, $\Ho_k(S^{2n-1})$ for $k = 1, 2$ are trivial, and hence the maps $\Phi_*$, $\Psi_*$ are isomorphisms. By the theorem of Eliashberg--Floer--McDuff \cite{McD91}, the filling $W$ is diffeomorphic to the ball $B^{2n}$, so $\Ho_2(W)$ is trivial. Moreover the second homology group $\Ho_2(\R^{n-1} \times \RP^1 \setminus \Int B^n)$ is isomorphic to $\Ho_2(\R^{n-1} \times \RP^1) \cong \Ho_2(\RP^1)$ which is trivial. (One way to see this is to consider the long exact sequence for the pair $(\R^{n-1} \times \RP^1, \R^{n-1} \times \RP^1 \setminus \Int B^n)$.) Therefore the second horizontal inclusion 
$$
i_* \oplus i_* \colon \Ho_2(L) \oplus \Ho_2(\R^{n-1} \times \RP^1 \setminus \Int B^n) \rightarrow \Ho_2(W) \oplus \Ho_2(\C^{n-1} \times \CP^1 \setminus \Int B^{2n})
$$
is vanishing. By the commutativity, we conclude that the inclusion
$$
i_* \colon \Ho_2(\widehat L) \longrightarrow \Ho_2(\widehat Z)
$$
is also trivial, as asserted.

To show the second assertion, we suppose on the contrary that $j_*[\sigma] = [\pt \times D^2_+]$ for some $J$-holomorphic sphere $\sigma\colon \CP^1 \rightarrow \widehat Z$. Consider the reflected $J$-holomorphic sphere $\overline \sigma\colon \CP^1 \rightarrow \widehat Z$ defined by
\begin{equation}\label{eq: doub_sphere}
\overline \sigma = \widehat{\mathcal{R}} \circ \sigma \circ \rho_{\CP^1}.
\end{equation}
Its homology class is given by $[\overline \sigma] = -\widehat{\mathcal{R}}_*[\sigma]$, and hence we have
$$
j_*[\overline \sigma]  = j_*(-\widehat{\mathcal{R}}_*[\sigma]) = -\widehat{\mathcal{R}}_* (j_*[\sigma]) = -\widehat{\mathcal{R}}_* ([\pt \times D_+^2]) = [\pt \times D_-^2].
$$
It follows that
$$
j_*([\sigma] + [\overline \sigma]) = [\pt \times D^2_+] + [\pt \times D^2_-] = j_*[\pt \times \CP^1].
$$
Since $j_*\colon \Ho_2(\widehat Z) \rightarrow \Ho_2(\widehat Z, \widehat L)$ is injective as in the above claim, we obtain
$$
[\pt \times \CP^1] = [\sigma] + [\overline \sigma].
$$
Consider the complex hypersurface $H$ in \eqref{eq: cpxhyp}. Note that the involution $\widehat{\mathcal{R}}$ on $\widehat Z$ restricts to~$H$. This implies that if the sphere $\sigma$ intersects $H$, then so does $\overline \sigma$.
Positivity of intersections implies that the intersection number $([\sigma] + [\overline \sigma]) \cdot [H]$ must be even, and this contradicts to $[\pt \times \CP^1] \cdot [H] = 1$.
\end{proof}

Let $\widehat{\mathcal{M}}_J$ be the space of $J$-holomorphic disks of $(\widehat{Z},\widehat{L})$ representing the homology class $[\pt\times D_+^2]\in \Ho_2(\widehat{Z},\widehat{L})$ and $\Aut(D^2)$ the group of biholomorphisms on $D^2$.
Unless we emphasize the choice of $J$, we suppress $J$ in the notation.
The moduli space $\widehat{\mathcal{M}}/\text{Aut}(D^2)$ admits the involution induced by $\widehat{\mathcal{R}}$
$$
\Xi [u]= [\widehat{\mathcal{R}}\circ u\circ \rho_{D^2}],
$$
where $\rho_{D^2}$ is complex conjugation of $D^2\subset \C$.
Achieving equivariant transversality for $\widehat{\mathcal{M}}$ is closely related to the phenomenon that $\Xi$ has no fixed point.

\begin{lemma}\label{lem: simple_nofixedpt} Let $n\ge 4$ and $J\in \J$.
\begin{enumerate}
\item Every $u\in \widehat{\mathcal{M}}$ is simple, i.e. the set of injective points of $u$,
$$
\{z\in D^2\mid du(z)\ne 0,\ u^{-1}(u(z))=\{z\}\}
$$
is dense.
\item The involution $\Xi$ of $\widehat{\mathcal{M}}/\Aut(D^2)$ has no fixed point.
\end{enumerate}
\end{lemma}

\begin{proof}
For the first assertion, let $u^\#\colon \CP^1 \rightarrow \widehat{Z}$ be the double of  $u\in\widehat{\mathcal{M}}$.
Then the simplicity of $u$ is equivalent to that of $u^\#$. Note that the homology class of $u^\#$ satisfies
$$
j_*[u^\#] = [\pt \times D^2_+] - \widehat{\mathcal{R}}_*[\pt \times D_+^2] = [\pt \times D^2_+] + [\pt \times D_-^2] = j_* [\pt \times \CP^1].
$$
Since $j_*$ is injective by Lemma~\ref{lem: notspherical}, we have 
$$
[u^\#] = [\pt \times \CP^1].
$$
Note that $[\pt \times \CP^1]$ is a primitive class in $\Ho_2(\widehat Z)$ since $[\pt \times \CP^1] \cdot [H] = 1$. Therefore $u^\#$ is a simple $J$-holomorphic curve, and so is $u$.


We now prove the second assertion by following \cite[Lemmata~3.4 and 3.5]{Kim21} and the proofs therein.
Arguing by contradiction we assume that there exists $[u]\in \widehat{\mathcal{M}}/\text{Aut}(D^2)$ such that $\Xi[u]=[u]$.
Then there is an anti-holomorphic involution $\rho$ of $D^2$ such that
\begin{itemize}
	\item the fixed point set of $\rho$ is a simple smooth arc with ends in $\p D^2$ and it divides the disk $D^2$ into two closed disks $\D_+$ and $\D_-$.
	\item $\widehat{\mathcal{R}}\circ u\circ \rho = u$.
\end{itemize}
Choose any biholomorphism $\phi\colon \Int D^2\to \Int\D_+$ that extends to a $C^0$-map from $D^2$ to $\D_+$.
The continuous maps from $(D^2,\p D^2)$ to $(\widehat{Z},\widehat{L})$,
$$
u_1=u\circ \phi,\qquad u_2=u\circ \rho\circ \phi\circ \rho
$$
extend to smooth $J$-holomorphic disks, still denoted by $u_1$ and $u_2$.
By construction, we obtain $[u]=[u_1]+[u_2]\in \Ho_2(\widehat{Z},\widehat{L})$.
Since
$$
\widehat{\mathcal{R}}\circ u_1 \circ \rho = \widehat{\mathcal{R}}\circ u \circ \phi \circ \rho = u \circ \rho \circ \phi \circ \rho = u_2
$$
and $\rho$ is orientation-reversing, we have $[u_2]=-\widehat{\mathcal{R}}_*[u_1]$.
Therefore, we deduce that
$$
j_*[u_1^\#] = [u_1] - \widehat{\mathcal{R}}_*[u_1] = [u_1]+[u_2] = [u].
$$
This contradicts to the fact that $[u]=[\pt \times D_+^2]$ is not $J$-spherical as proved in Lemma~\ref{lem: notspherical}.
\end{proof}

We are in position to prove Proposition~\ref{prop: eqregular}.

\begin{proof}[Proof of Proposition~\ref{prop: eqregular}]
For a standard holomorphic disk $u$, its double $u^\#$ is a simple $J$-holomorphic sphere in $\R^{n-1}\times \C P^1\setminus \Int B^{2n}\subset \widehat{Z}$ whose image is $\{v_0\}\times \C P^1$ for some $v_0\in \R^{n-1}$.
It is known, see for example \cite[Corollary~3.3.5]{MSJcurve}, that
the linearized Cauchy--Riemann operator $\mathbf{D}_{u^\#}$ of $u^\#$ between suitable Banach spaces is surjective as $J=J_0\oplus j$ on $\C^{n-1}\times \C P^1\setminus \Int B^{2n}$. 
By \cite[Proposition~on~pg.\ 158]{HLS} or \cite[Section~5.3]{Kim21}, the operator $\mathbf{D}_u$ corresponding to the standard disk $u$ is surjective as well, and this holds for all $J\in \J$.
Combining the doubling trick with the maximum principle \cite[(i) in Lemma~2.1]{BarGeiZeh19}, we deduce that every $u\in \widehat{\mathcal{M}}$ is of a standard disk or must intersect the interior of $W$.
Since $\Xi$ has no fixed point by Lemma~\ref{lem: simple_nofixedpt}, it follows from equivariant transversality theorem~\cite[Section~3.3]{Kim21} that there exists a Baire subset $\J^{\reg}\subset \J$ such that for every $J\in\J^{\reg}$ the operator $\mathbf{D}_u$ of any disk $u\in \widehat{\mathcal{M}}$ intersecting $\Int W$ is surjective.
Therefore, for $J\in \J^{\reg}$ the space $\widehat{\mathcal{M}}$ is a smooth manifold of dimension
$$
\dim \widehat{\mathcal{M}} = n\cdot\chi(D^2)+\mu_{\widehat{L}}(u) = n+2,
$$
where $\mu_{\widehat{L}}(u)$ is the Maslov index of $u\in \widehat{\mathcal{M}}$.
To complete the proof, we shall show that there exists a Baire subset $\widehat{\J^{\reg}}$ of $\J$ such that for every $J\in \widehat{\J^{\reg}}$ we have $J\in \J^{\reg}$ and the evaluation map
$$
\ev_{\mathbf{w}} \colon \widehat{\mathcal{M}} \to \widehat{L}\times \widehat{L}\times \widehat{L},\qquad \ev_{\mathbf{w}}(u)=(u(w_1),u(w_2),u(w_\infty))
$$
is transverse to the submanifold
$$
K:=(\R^{n-1}\times \{z_1\})\times (\R^{n-1}\times \{z_2\})\times (\R^{n-1}\times \{z_\infty\}),
$$
since then $\mathcal{M}=\ev_{\mathbf{w}}^{-1}(K)$ is a smooth manifold of dimension
$$
\dim \mathcal{M}=\dim \widehat{\mathcal{M}}-\codim K = n-1.
$$
{\bf Claim.} \emph{Every point in $\widehat{L}\times \widehat{L}\times \widehat{L}$ of pairwise distinct points in $\widehat{L}$ is a regular value of the universal evaluation map
$$
\ev_{\univ}\colon \mathcal{N} \to \widehat{L}\times \widehat{L}\times \widehat{L},\qquad \ev_{\univ}(u,J)=\ev_{\mathbf{w}}(u),
$$
where $\mathcal{N}=\{(u,J)\mid J\in \J^\ell,\ u\in \widehat{\mathcal{M}}_J\}$ is the universal moduli space and $\J^\ell$ consists of $J\in \J$ of $C^\ell$-classes.}\vspace{0.2cm}

Let $(u,J)\in \mathcal{N}$ such that $\ev_{\mathbf{w}}(u,J)=(u(w_1),u(w_2),u(w_\infty))$ is a tuple of pairwise distinct points.
For any tuple of vectors $v_i\in T_{u(w_i)}\widehat{L}$ choose a smooth function $F\colon \widehat{Z}\to \R$ such that $X_F(u(w_i))+v_i=0$, where $X_F$ is the Hamiltonian vector field of $F$.
Using \cite[Exercise~3.1.4]{MSJcurve}, we know that $\mathbf{D}_u(X_F(u))=\frac{1}{2}(\mathcal{L}_{X_F}J)(u)du\circ j$. 
Since $\widehat{\mathcal{R}}^*X_F=X_F$ and $J$ is $\widehat{\mathcal{R}}$-anti-invariant, we have $\widehat{\mathcal{R}}^*(\mathcal{L}_{X_F}J)=\mathcal{L}_{\widehat{\mathcal{R}}^*X_F}\widehat{\mathcal{R}}^*J=-\mathcal{L}_{X_F}J$, yielding that $(-X_F(u),\mathcal{L}_{X_F}J)\in T_{(u,J)}\mathcal{N}$ and
$$
d\ev_{\univ}(u,J)\big(-X_F(u),\mathcal{L}_{X_F}J\big)=(v_1,v_2,v_\infty).
$$
See \cite[Section~3.4]{MSJcurve} for details.
This shows the claim as asserted.

Since every point in $K$ is of pairwise distinct points, $\ev_{\univ}$ is transverse to $K$, and hence the universal moduli space $\ev_{\univ}^{-1}(K)$
is a Banach manifold.
The rest of the proof is a standard argument, namely that we use the Sard--Smale theorem and Taubes' argument to the projection map $\pi\colon \ev_{\univ}^{-1}(K) \to \J^\ell$.
We refer to \cite[Theorem~3.1.6 and Proposition~3.4.2]{MSJcurve} for more details.

In order to orient the moduli space $\mathcal{M}$, it suffices to show that $\widehat{L}$ and $\widehat{\mathcal{M}}$ are orientable.
By Remark~\ref{rem: smith}, we know $\Ho^k(L;\Z_2)=0$ for all $k\ge 1$.
In particular, $L$ is orientable, and so is $\widehat{L}$.
By the Mayer--Vietoris sequence for the decomposition $\widehat{L}=L\cup_{\mathcal{L}_0} (\R^{n-1}\times \R P^1 \setminus \Int B^n)$, one can show that $\Ho^2(\widehat L;\Z_2)=0$, and hence $\widehat{L}$ is spin.
A standard construction in \cite[Section~8.1]{FOOObook2} yields an orientation on $\widehat{\mathcal{M}}$.
\end{proof}

\begin{remark}\label{rem: endofmoduli}
As in \cite[Lemma~2.1]{BarGeiZeh19}, we can apply the maximum principle in the $\C^{n-1}$-direction of $\C^{n-1}\times \CP^1 \setminus \Int B^{2n} \subset \widehat Z$ to $J$-holomorphic spheres which are the doubles of $J$-holomorphic disks $u \in \mathcal{M}$.
We see that every $u \in \mathcal{M}$ whose double intersects the region $\{(v,z)\in \C^{n-1}\times \CP^1 \setminus \Int B^{2n} \subset \widehat Z \mid \|v\| > 1\}$ is a standard disk, i.e. $u(z) = (v_0, \tau(z))$ for some $v_0 \in \R^{n-1}$ and a biholomorphism $\tau$ from $D^2$ to $D_+^2$.
It follows that the end of the moduli space $\mathcal{M}$,
$$
\{u\in \mathcal{M} \mid \text{$u(z)=(v_0,\tau(z))$ is standard with $v_0\in \R^{n-1}$ and $\|v_0\| > 1$}\},
$$
is naturally identified with the space $\{v \in \R^{n-1}\mid \|v\| > 1\}$.
We orient the moduli space $\mathcal{M}$ so that this identification is orientation-preserving.
\end{remark}

\subsection{Evaluation map}\label{sec: evalutation}
From now on, let $n\ge 4$ and choose generic $J\in \J$ as in Proposition~\ref{prop: eqregular}.
Consider the evaluation map $\widehat \ev\colon \mathcal{M} \times \p D^2 \rightarrow \widehat L$ given by $\widehat \ev(u, z) = u(z)$. 

\begin{proposition}\label{prop: proper}
The evaluation map $ \widehat \ev\colon \mathcal{M} \times \p D^2 \rightarrow \widehat L$ is proper and is of degree $1$.
\end{proposition}

\begin{proof}
For properness, it suffices to show in view of Remark~\ref{rem: endofmoduli} that if $u_{\infty}$ is a stable map described in \cite{FraZeh15} which is the Gromov-limit of a sequence $\{u_{\nu}\}$ in $\mathcal{M}$, then $u_{\infty} \in \mathcal{M}$. The map $u_{\infty}$ consists of finitely many $J$-holomorphic disks and spheres attached at finitely many points. Denote by 
$$
u^1, \dots, u^{N_1}, \sigma^1, \dots, \sigma^{N_2}
$$ 
those disks $u^j$ and spheres $\sigma^k$.  We know from \cite[Theorem~1.1]{FraZeh15} that the homology class of $[u_{\infty}]$ is still $[\pt \times D^2_+]$.

By doubling each of $u^j$'s and $\sigma^k$'s, we obtain a collection $u_{\infty}^\#$ of $J$-holomorphic spheres, which consists of  
$$
(u^1)^\#, \dots, (u^{N_1})^\#, \sigma^1, \overline \sigma^1, \dots, \sigma^{N_2}, \overline \sigma^{N_2}
$$
where $\overline \sigma^k$ is defined as in \eqref{eq: doub_sphere}. 
See Figure~\ref{fig: limit_tree}.

\begin{figure}[h]
\begin{tikzpicture}[scale=0.6]

\begin{scope}[yshift=-0.5cm]
\draw (-0.5, 0) arc [x radius = 1.2cm, y radius = 1.3cm, start angle = 0,
  end angle = 180];
\draw (-0.5, 0) arc [x radius = 1.2cm, y radius = 0.2cm, start angle = 0,
  end angle = -180];
\draw [dashed] (-0.5, 0) arc [x radius = 1.2cm, y radius = 0.2cm, start angle = 0,
  end angle = 180];
\draw (1.9, 0) arc [x radius = 1.2cm, y radius = 1.3cm, start angle = 0,
  end angle = 180];
\draw (1.9, 0) arc [x radius = 1.2cm, y radius = 0.2cm, start angle = 0,
  end angle = -180];
\draw [dashed] (1.9, 0) arc [x radius = 1.2cm, y radius = 0.2cm, start angle = 0,
  end angle = 180];
  
\draw (2.05, 1.3) circle [radius=0.6];
\draw (2.65, 1.3) arc [x radius = 0.6cm, y radius = 0.1cm, start angle = 0,
  end angle = -180];
\draw [dashed] (2.65, 1.3) arc [x radius = 0.6cm, y radius = 0.1cm, start angle = 0,
  end angle = 180];

\begin{scope}[xshift=1.2cm, yshift=-0.2cm]
\draw (2.05, 1.4) circle [radius=0.6];
\draw (2.65, 1.4) arc [x radius = 0.6cm, y radius = 0.1cm, start angle = 0,
  end angle = -180];
\draw [dashed] (2.65, 1.4) arc [x radius = 0.6cm, y radius = 0.1cm, start angle = 0,
  end angle = 180];	
\end{scope}

\begin{scope}[xshift=-1.6cm, yshift=0.5cm]
\draw (2.05, 1.4) circle [radius=0.6];
\draw (2.65, 1.4) arc [x radius = 0.6cm, y radius = 0.1cm, start angle = 0,
  end angle = -180];
\draw [dashed] (2.65, 1.4) arc [x radius = 0.6cm, y radius = 0.1cm, start angle = 0,
  end angle = 180];	
\end{scope}
\end{scope}

\draw [thick,->] (4.7, 0)--(6.5, 0);
\node at (5.5, 0.5) {double};

\begin{scope}[xshift=10cm]
\draw (-0.5, 0) arc [x radius = 1.2cm, y radius = 1.3cm, start angle = 0,
  end angle = 180];
\draw (-0.5, 0) arc [x radius = 1.2cm, y radius = 0.2cm, start angle = 0,
  end angle = -180];
\draw [dashed] (-0.5, 0) arc [x radius = 1.2cm, y radius = 0.2cm, start angle = 0,
  end angle = 180];
\draw (1.9, 0) arc [x radius = 1.2cm, y radius = 1.3cm, start angle = 0,
  end angle = 180];
\draw (1.9, 0) arc [x radius = 1.2cm, y radius = 0.2cm, start angle = 0,
  end angle = -180];
\draw [dashed] (1.9, 0) arc [x radius = 1.2cm, y radius = 0.2cm, start angle = 0,
  end angle = 180];
  
\draw (2.05, 1.3) circle [radius=0.6];
\draw (2.65, 1.3) arc [x radius = 0.6cm, y radius = 0.1cm, start angle = 0,
  end angle = -180];
\draw [dashed] (2.65, 1.3) arc [x radius = 0.6cm, y radius = 0.1cm, start angle = 0,
  end angle = 180];

\begin{scope}[xshift=1.2cm, yshift=-0.2cm]
\draw (2.05, 1.4) circle [radius=0.6];
\draw (2.65, 1.4) arc [x radius = 0.6cm, y radius = 0.1cm, start angle = 0,
  end angle = -180];
\draw [dashed] (2.65, 1.4) arc [x radius = 0.6cm, y radius = 0.1cm, start angle = 0,
  end angle = 180];	
\end{scope}

\begin{scope}[xshift=-1.6cm, yshift=0.5cm]
\draw (2.05, 1.4) circle [radius=0.6];
\draw (2.65, 1.4) arc [x radius = 0.6cm, y radius = 0.1cm, start angle = 0,
  end angle = -180];
\draw [dashed] (2.65, 1.4) arc [x radius = 0.6cm, y radius = 0.1cm, start angle = 0,
  end angle = 180];	
\end{scope}
	
\end{scope}

\begin{scope}[xshift=10cm, yscale=-1]
\draw (-0.5, 0) arc [x radius = 1.2cm, y radius = 1.3cm, start angle = 0,
  end angle = 180];
  end angle = -180];
\draw [dashed] (-0.5, 0) arc [x radius = 1.2cm, y radius = 0.2cm, start angle = 0,
  end angle = 180];
\draw (1.9, 0) arc [x radius = 1.2cm, y radius = 1.3cm, start angle = 0,
  end angle = 180];
  end angle = -180];
\draw [dashed] (1.9, 0) arc [x radius = 1.2cm, y radius = 0.2cm, start angle = 0,
  end angle = 180];
  
\draw (2.05, 1.3) circle [radius=0.6];
\draw [dashed] (2.65, 1.3) arc [x radius = 0.6cm, y radius = 0.1cm, start angle = 0,
  end angle = -180];
\draw (2.65, 1.3) arc [x radius = 0.6cm, y radius = 0.1cm, start angle = 0,
  end angle = 180];

\begin{scope}[xshift=1.2cm, yshift=-0.2cm]
\draw (2.05, 1.4) circle [radius=0.6];
\draw [dashed] (2.65, 1.4) arc [x radius = 0.6cm, y radius = 0.1cm, start angle = 0,
  end angle = -180];
\draw  (2.65, 1.4) arc [x radius = 0.6cm, y radius = 0.1cm, start angle = 0,
  end angle = 180];	
\end{scope}

\begin{scope}[xshift=-1.6cm, yshift=0.5cm]
\draw (2.05, 1.4) circle [radius=0.6];
\draw [dashed] (2.65, 1.4) arc [x radius = 0.6cm, y radius = 0.1cm, start angle = 0,
  end angle = -180];
\draw (2.65, 1.4) arc [x radius = 0.6cm, y radius = 0.1cm, start angle = 0,
  end angle = 180];	
\end{scope}	
\end{scope}

\end{tikzpicture}
\caption{A Gromov-limit $u_\infty$ and its double $u_\infty^\#$}
\label{fig: limit_tree}
\end{figure}
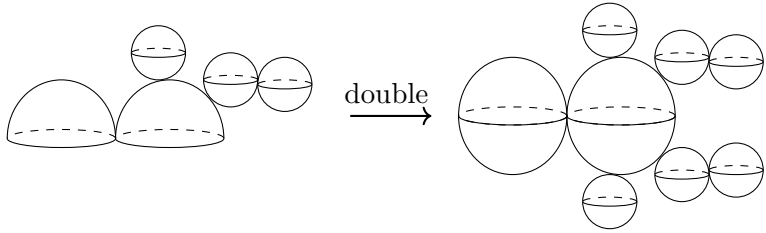
\noindent
The homology class of $u_{\infty}^\#$ then satisfies
$$
j_*[u_{\infty}^\#] = [u_{\infty}] -\widehat{\mathcal{R}}_* [u_{\infty}] = [\pt \times D^2_+] + [\pt \times D^2_-] = j_*[\pt \times \CP^1].
$$
Since $j_*$ is injective by Lemma~\ref{lem: notspherical}, it follows that
$$
[u_{\infty}^\#] = [\pt \times \CP^1] \in \Ho_2(\widehat Z).
$$
The intersection number with the complex hypersurface $H$ is given by
\begin{align}\label{eq: intersectionstabledouble}
\begin{split}
[u_{\infty}^\#] \cdot [H] &= \left(\sum_{j=1}^{N_1} [(u^j)^\#] + \sum_{k=1}^{N_2} \big([\sigma^k]+ [\overline \sigma^k]\big)\right)  \cdot [H] \\ 
&=  [\pt \times \CP^1]  \cdot [H] =1.
\end{split}
\end{align}
Observe that $[\sigma^k] \cdot [H] =  [\overline \sigma^k] \cdot [H]$ for each $1 \leq k \leq N_2$ since $H$ is invariant under~$\widehat{\mathcal{R}}$. 
By positivity of intersections, we must have $[\sigma^k] \cdot [H]  = [\overline \sigma^k] \cdot [H] = 0$ in view of \eqref{eq: intersectionstabledouble}. 
Since $\sigma^k$ and $\overline \sigma^k$ are $J$-holomorphic spheres in $\widehat Z$, applying \cite[Lemma~2.1]{BarGeiZeh19} and the assumption that $W$ is symplectically aspherical, we conclude that they are constant by the same argument as in \cite[Proposition~2.3]{BarGeiZeh19}.

Now, by positivity of intersections again, we may assume that 
$$
[(u^1)^\#] \cdot [H] = 1, \quad [(u^j)^\#] \cdot [H] = 0\quad \text{for $2\le j \le N_1$}.
$$
As before, $(u^j)^\#$ is constant for $2 \leq j \leq N_1$, and this implies that the original disk $u^j$ is constant as well. We therefore conclude that $u_{\infty}$ consists of a single $J$-holomorphic disk, and this means  $u_{\infty} \in \mathcal{M}$, as desired.

Notice from Remark~\ref{rem: endofmoduli} that the preimage $\widehat \ev^{-1}(v, z)$ of any point $(v, z) \subset \R^{n-1} \times \RP^1 \setminus \Int B^n$ with $\|v\| > 1$ consists of a single element coming from a standard disk.
Therefore the degree of $\widehat \ev$ is equal to $1$.
\end{proof}

\subsection{Proof of Theorem~\ref{thm: main2}} Observe that the evaluation map $\widehat \ev\colon \mathcal{M} \times \p D^2 \rightarrow \widehat L$ restricts~to 
$$
\ev \colon \mathcal{M} \times (\p D^2 \setminus\{w_\infty\}) \longrightarrow \widehat L \setminus (\R^{n-1} \times \{\infty\}).
$$
Indeed, for $u \in \mathcal{M}$ the double $u^\#\colon \CP^1 \rightarrow \widehat Z$ intersects only once with the complex hypersurface $H$ by positivity of intersections, and therefore the image $u(\p D^2 \setminus \{w_\infty\})$ does not intersect $\R^{n-1}\times \{\infty\} \subset H$. Moreover, $\ev$ is a proper map of degree~1 by Proposition~\ref{prop: proper}.

Now consider the following commutative diagram:
$$
\begin{tikzcd}
\mathcal{M} \times (\p D^2 \setminus \{w_\infty\}) \arrow{r}{\ev}  & \widehat L \setminus (\R^{n-1} \times \{\infty\}) \\%
\mathcal{M} \times \{w_2\} \arrow{r}{\ev} \arrow[hookrightarrow]{u}{\text{h.e.}}& \R^{n-1} \times \{z_2\} \arrow[hookrightarrow]{u}{i}
\end{tikzcd}
$$
Here h.e. means a homotopy equivalence.
Since the upper evaluation map $\ev$ is proper and of degree 1, it is straightforward to adapt the argument in \cite[Proposition~2.4]{BarGeiZeh19} to show that the induced map on homology groups is surjective. Furthermore the same argument as in \cite[Section~2.5]{BarGeiZeh19} using covering spaces shows that the induced map on the fundamental groups is also surjective. It follows that the inclusion
$$
i\colon \R^{n-1} \times \{w_2\} \longhookrightarrow \widehat L \setminus (\R^{n-1} \times \{\infty\})
$$
likewise induces surjective maps on the homology groups and fundamental groups by the commutativity of the diagram. 
Since $\widehat L \setminus (\R^{n-1} \times \{\infty\}) \overset{\text{h.e.}}{\simeq} L$, we conclude that
$\pi_1(L) = 0$ and $\tilde \Ho_k(L) = 0$ for all $k\ge 0$. For $n \geq 6$, the $h$-cobordism theorem \cite[Proposition~A in pg.\ 108]{Milnor} tells us that $L$ is diffeomorphic to the ball $B^n$. This finishes the proof of Theorem~\ref{thm: main2}.

\subsection*{Acknowledgement}
The authors cordially thank Georgios Dimitroglou Rizell for helpful comments on the first version of the paper. 
JK is supported by the National Research Foundation of Korea(NRF) grant funded by the Korea government(MSIT) (No. 2022R1F1A1074587) and the Sogang University Research Grant of 202210022.01.
MK is supported by the National Research Foundation of Korea(NRF) grant funded by the Korea government(MSIT) (No. NRF-2021R1F1A1060118).
A part of this work was done while the authors visited KIAS.
They are grateful for its warm hospitality.

\bibliographystyle{abbrv}
\bibliography{mybibfile}

\end{document}